\documentclass[11pt, a4paper]{article}

\usepackage{graphicx}

\usepackage[T1]{fontenc}
\usepackage{multirow}
\usepackage{float}
\usepackage{amsmath}
\usepackage{amsfonts}
\usepackage{amssymb}
\usepackage{amsthm}
\usepackage{authblk}

\DeclareMathOperator{\GL}{GL}
\DeclareMathOperator{\Mat}{Mat}
\DeclareMathOperator{\lcm}{lcm}
\DeclareMathOperator{\ord}{ord}
\DeclareMathOperator{\diag}{diag}

\newtheorem{theorem}{Theorem}[section]
\newtheorem{corollary}[theorem]{Corollary}
\newtheorem{lemma}[theorem]{Lemma}
\newtheorem{proposition}[theorem]{Proposition}
\theoremstyle{definition}
 \newtheorem{definition}[theorem]{Definition}
\theoremstyle{remark}
 \newtheorem{remark}[theorem]{Remark}
\theoremstyle{remark}
 \newtheorem{example}[theorem]{Example}
\theoremstyle{definition}

\numberwithin{equation}{section}

\begin{document}
\begin{sloppy}

\title{Carmichael numbers for $\GL(m)$}
\author{Eugene Karolinsky}
\author{Dmytro Seliutin}
\date{}

\affil{\small{Department of Pure Mathematics, Kharkiv National University, Ukraine}}

\maketitle
\begin{abstract}
We propose a generalization of Carmichael numbers, where the multiplicative group $\mathbb G_\mathrm{m} = \GL(1)$ is replaced by $\GL(m)$ for $m\geq 2$. We prove basic properties of these families of numbers and give some examples.


\end{abstract}

\section{Introduction}
Recall that a composite number $n\in\mathbb N$ is called \emph{Carmichael} if $a^{n-1}=1$ for any $a\in(\mathbb Z/n\mathbb Z)^\times$. In other words, Carmichael numbers are Fermat pseudoprimes to all values of $a$ (coprime to $n$).

Recently, various generalizations and analogues of Carmichael numbers were proposed, see, e.g., \cite{Howe, McIntosh, Steele} and references therein. In this paper, we introduce a different analogue of Carmichael numbers, where the multiplicative group $\mathbb G_\mathrm{m} = \GL(1)$ is replaced by $\GL(m)$ for $m\geq 2$. Namely, we start with the exponent $K_m(p)$ of the group $\GL(p,\mathbb F_p)$, extrapolate it naturally to all naturals as $K_m(n)$, and then define a composite number $n\in\mathbb N$ to be $m$-Carmichael if $A^{K_m(n)}=I$ for all $A\in\GL(m, \mathbb Z/n\mathbb Z)$. Thus, ``classical'' Carmichael numbers essentially are recovered as $1$-Carmichael.

We study basic properties of $m$-Carmichael numbers, including an analogue of the Korselt's criterion for a number to be  Carmichael in terms of its prime divisors. This criterion appears practical for numbers of reasonable size, and we compute all $m$-Carmichael numbers less or equal than $10^5$ for $2\leq m\leq10$. We also describe the structure of $m$-Carmichael numbers with given prime factors.

Some properties of $m$-Carmichael numbers for $m\geq 2$ appear to be rather different from those of ``classical'' Carmichael numbers. Namely, $m$-Carmichael numbers for $m\geq 2$ need not to be squarefree or odd. Moreover, all prime powers are $m$-Carmichael for $m\geq 2$. Possible explanation of these phenomena is the fact that the groups $\GL(m)$ for $m\geq2$ contain (many copies of) the additive group $\mathbb G_\mathrm{a}$.

The paper is organized as follows. In Section \ref{sec_2}, we define $m$-Carmichael numbers and discuss their basic properties. The main result of the paper is Theorem \ref{thm_Korselt}, an analogue of the Korselt's criterion for $m$-Carmichael numbers. In Section \ref{sec_2e}, we consider the distribution of $m$-Carmichael numbers with prescribed prime factors, giving several particular examples, and summarizing the general pattern in Theorem \ref{thm_P}. In Section \ref{sec_4}, we list some open questions and discuss possible generalizations. In Appendix, we describe our computations of relatively small $m$-Carmichael numbers.

Throughout this paper, we denote by $p$ a prime number. In particular, $\prod_{p\mid n}$ means a product taken over all prime divisors of $n$.

\section{$m$-Carmichael numbers}\label{sec_2}
Let $\Phi_k(X)$ be the $k$th cyclotomic polynomial. The following proposition is well known, but for the reader's convenience we present a proof.

\begin{proposition}
If $a\in \mathbb Z$, then
\[
\lcm(a-1,a^2-1,\ldots,a^m-1)=\prod_{k=1}^m\Phi_k(a).
\]
\end{proposition}

\begin{proof}
We proceed by induction with an obvious base. We have
\begin{gather*}
\lcm(a-1,a^2-1,\ldots,a^m-1)=\lcm\left(\prod_{k=1}^{m-1}\Phi_k(a),a^m-1\right)=\\
\left(\prod_{d\mid m, d<m}\Phi_d(a)\right)\cdot\lcm\left(\prod_{k\nmid m, k<m}\Phi_k(a),\Phi_m(a)\right).
\end{gather*}
By \cite[Theorem 5]{Ge} we have $\gcd(\Phi_k(a),\Phi_l(a))=1$ unless $\frac{k}{l}$ is a prime power. Therefore,
\[
\gcd\left(\prod_{k\nmid m, k<m}\Phi_k(a),\Phi_m(a)\right)=1,
\]
and
\[
\lcm\left(\prod_{k\nmid m, k<m}\Phi_k(a),\Phi_m(a)\right)=\Phi_m(a)\prod_{k\nmid m, k<m}\Phi_k(a),
\]
which finishes the proof.
\end{proof}

Recall that the \emph{exponent} of a (finite) group $G$ is the least common multiple of the orders of elements of $G$.

\begin{theorem}\cite{Marshall, Niven}\label{fermat_matr}
The exponent of $\GL(m,\mathbb F_p)$ equals
\[
p^{\lceil\log_pm\rceil}\lcm(p-1,p^2-1,\ldots,p^m-1)=p^{\lceil\log_pm\rceil}\prod_{k=1}^m\Phi_k(p).
\]
\qed
\end{theorem}

From now on we assume that $m\geq2$.

Let us introduce the following notation:

\[
D_m(n)=\prod_{k=1}^m\Phi_k(n),
\]
\[
\nabla_m(n)=\prod_{p\mid n}p^{\lceil\log_pm\rceil-1},
\]
\[
K_m(n)=n\nabla_m(n)D_m(n).
\]
In this notation, the exponent of $\GL(m,\mathbb F_p)$ equals $K_m(p)$.

Notice also that if $p\geq m$, then $p^{\lceil\log_pm\rceil-1}=1$. Therefore,
\[
\nabla_m(n)=\prod_{p\mid n,\ p<m}p^{\lceil\log_pm\rceil-1}.
\]

\begin{example}
1) We have $\nabla_2(n)=1$, $D_2(n)=(n-1)(n+1)$, thus $K_2(n)=n(n-1)(n+1)$.

2) We have $\nabla_3(n)=1$ for $n$ odd, $\nabla_3(n)=2$ for $n$ even, and $D_3(n)=(n-1)(n+1)(n^2+n+1)$. Therefore  $K_3(n)=n(n-1)(n+1)(n^2+n+1)$ for $n$ odd and $K_3(n)=2n(n-1)(n+1)(n^2+n+1)$ for $n$ even.
\end{example}

\begin{definition}
A composite number $n\in\mathbb N$ is called an \emph{$m$-Carmichael number} if $A^{K_m(n)}=I$ for all $A\in\GL(m,\mathbb Z/n\mathbb Z)$.
\end{definition}

First, we show that any prime power is an $m$-Carmichael number. For this purpose, we need two simple lemmas.

\begin{lemma}\label{lemma_cycl_div}
If $a\in\mathbb Z$, $k\in \mathbb N$, then $D_m(a)\mid D_m(a^k)$.
\end{lemma}

\begin{proof}
Consider $D_m(X)=\prod_{k=1}^m\Phi_k(X)\in\mathbb Z[X]$. Then, since all roots of $D_m(X)$ are simple, and each root of $D_m(X)$ is a root of $D_m(X^k)$, we have $D_m(X)\mid D_m(X^k)$. Since the polynomial $D_m(X)$ is monic, this implies the lemma.
\end{proof}

\begin{lemma}\label{lemma_matr_mod_p_pow}
Let $B\in\Mat(m,\mathbb Z)$, $B\equiv I\mod p$. Then for any $k\in\mathbb N$ we have $B^{p^{k-1}}\equiv I\mod p^k$.
\end{lemma}

\begin{proof}
By the binomial theorem, we have $B^p\equiv I\mod p^2$. Then use induction.
\end{proof}

\begin{proposition}\label{prop_prime_pow}
If $k\in\mathbb N$, $k>1$, then $p^k$ is an $m$-Carmichael number.
\end{proposition}

\begin{proof}
Let $A\in\Mat(m,\mathbb Z)$, $\gcd(\det A, p)=1$. By Theorem \ref{fermat_matr}, we have $B:=A^{K_m(p)}\equiv I\mod p$. Therefore, by Lemma \ref{lemma_matr_mod_p_pow} we have $B^{p^{k-1}}\equiv I\mod p^k$. Since $\nabla_m(p^k)=\nabla_m(p)$ and, by Lemma \ref{lemma_cycl_div}, $D_m(p)\mid D_m(p^k)$, the equation $B^{p^{k-1}}=A^{p^k\nabla_m(p)D_m(p)}\equiv I\mod p^k$ implies $A^{K_m(p^k)}=A^{p^k\nabla_m(p^k)D_m(p^k)}\equiv I\mod p^k$.
\end{proof}

Now, we present the main theorem of the paper, a Korselt type criterion for a number to be $m$-Carmichael.

\begin{theorem}\label{thm_Korselt}
Let $n\in\mathbb N$ be composite. The following are equivalent:

(1) $n$ is an $m$-Carmichael number,

(2) if $p\mid n$, then $D_m(p)\mid K_m(n)$.
\end{theorem}

\begin{proof}
1) Let $D_m(p)\mid K_m(n)$ for all $p\mid n$. Since $p^{\ord_p(n)}\nabla_m(p)\mid K_m(n)$ and $\gcd(p^{\ord_p(n)}\nabla_m(p),D_m(p))=1$, we also have $p^{\ord_p(n)}\nabla_m(p)D_m(p)\mid K_m(n)$ for all $p\mid n$.

Now consider $A\in\Mat(m,\mathbb Z)$, $\gcd(\det A, n)=1$. By Theorem \ref{fermat_matr}, we have $A^{p\nabla_m(p)D_m(p)}\equiv I\mod p$ for all $p\mid n$. By Lemma \ref{lemma_matr_mod_p_pow}, this implies $A^{p^{\ord_p(n)}\nabla_m(p)D_m(p)}\equiv I\mod p^{\ord_p(n)}$, and thus, $A^{K_m(n)}\equiv I\mod p^{\ord_p(n)}$ for all $p\mid n$. By the Chinese remainder theorem, this implies $A^{K_m(n)}\equiv I\mod n$. Therefore, $n$ is an $m$-Carmichael number.

2) Conversely, assume that for some $p\mid n$ we have $D_m(p)\nmid K_m(n)$. Since $D_m(p)=\lcm(p-1,p^2-1,\ldots,p^m-1)$, there exists $k\in\{1, 2,\ldots, m\}$ such that $p^k-1\nmid K_m(n)$.

We construct an $A\in\GL(m,\mathbb Z/n\mathbb Z)$ of order $p^k-1$. Therefore, $A^{K_m(n)}\neq I$, and $n$ is not $m$-Carmichael.

To this end, let $\alpha$ be a generator of the cyclic group $\mathbb F_{p^k}^\times$. Consider the polynomial $(X-\alpha)(X-\alpha^p)\ldots(X-\alpha^{p^{k-1}})\in \mathbb F_p[X]$, and let $B\in\GL(k,\mathbb F_p)$ be its accompanying matrix. Then $B$ is of order $p^k-1$, and the same is true for $C=\diag(B,I)\in\GL(m,\mathbb F_p)$. Lift $C$ to an element of $\Mat(m,\mathbb Z)$, so in particular $C^{p^k-1}\equiv I\mod p$. By Lemma \ref{lemma_matr_mod_p_pow}, we have $(C^{p^{\ord_p(n)-1}})^{p^k-1}\equiv I\mod p^{\ord_p(n)}$. Moreover, since $\gcd(p^{\ord_p(n)}, p^k-1)=1$, we see that $C^{p^{\ord_p(n)-1}}\mod p^{\ord_p(n)}$ is also of order $p^k-1$. Finally, by the Chinese remainder theorem, consider $A\in\GL(m,\mathbb Z/n\mathbb Z)$ such that $A\equiv C^{p^{\ord_p(n)-1}}\mod p^{\ord_p(n)}$ and, for example, $A\equiv I\mod \frac{n}{p^{\ord_p(n)}}$. By construction, $A$ is of order $p^k-1$, which finishes the proof.
\end{proof}

\begin{remark}
Proposition \ref{prop_prime_pow} also easily follows from Theorem \ref{thm_Korselt}.
\end{remark}

Moreover, applying Theorem \ref{thm_Korselt} and Lemma \ref{lemma_cycl_div}, we get

\begin{corollary}\label{cor_powers}
If $n$ is an $m$-Carmichael number, and $k\in\mathbb N$, then $n^k$ is also an $m$-Carmichael number.\qed
\end{corollary}

Finally, we present one necessary condition for a number to be $m$-Car\-mi\-chael.

\begin{proposition}
Assume that $n$ is an $m$-Carmichael number. Then $n\not\equiv 2\mod 4$.
\end{proposition}

\begin{proof}
Let $n\in\mathbb N$ be composite, $n\equiv 2\mod 4$. Then $D_m(n)$ is odd, and thus $\ord_2(K_m(n))=\lceil\log_2m\rceil$.

On the other hand, consider an odd $p\mid n$. Since $\Phi_{2^k}(p)=p^{2^{k-1}}+1$ is even for $k\geq 1$, and $8\mid\Phi_1(p)\Phi_2(p)=p^2-1$, we have $\ord_2(D_m(p))\geq\lfloor\log_2m\rfloor+2>\lceil\log_2m\rceil$. Thus, $D_m(p)\nmid K_m(n)$, and $n$ is not $m$-Carmichael.
\end{proof}

\section{$m$-Carmichael numbers having prescribed prime factors}\label{sec_2e}
The divisibility condition in Theorem \ref{thm_Korselt} for small values of $m$ is transparent enough to find some infinite families of $m$-Carmichael numbers (apart of prime powers).

Let us start with $m=2$. Denote by $d_2(p, n)$ the condition $D_2(p)\mid K_2(n)$, i.e., $p^2-1\mid n(n^2-1)$. We have

\begin{itemize}
\item $d_2(2, n)$ is $3\mid n(n^2-1)$, satisfied for all $n$.
\item $d_2(3, n)$ is $8\mid n(n^2-1)$, satisfied if and only if $n$ is odd or $8\mid n$.
\item $d_2(5, n)$ is $3\cdot8\mid n(n^2-1)$, again satisfied if and only if $n$ is odd or $8\mid n$.
\item $d_2(7, n)$ is $3\cdot16\mid n(n^2-1)$, satisfied if and only if $n\equiv \pm1\mod 8$ or $16\mid n$.
\item $d_2(11, n)$ is $3\cdot5\cdot8\mid n(n^2-1)$, satisfied if and only if $d_2(5, n)$ and $5\mid n(n^2-1)$ are satisfied; the latter is satisfied if and only if $n\equiv 0,\pm1\mod 5$.
\end{itemize}

Using the above the following propositions are proved by a direct application of Theorem \ref{thm_Korselt}.

\begin{proposition}\label{prop_2_Carm_7}
Let $n\in\mathbb N$ be a composite $7$-smooth number which is not a prime power. Then $n$ is $2$-Carmichael if and only if $n$ belongs to one of the following families:

1) $n=2^k\cdot3^l\cdot5^r$, where $k\geq3$,

2) $n=2^k\cdot3^l\cdot5^r\cdot7^s$, where $k\geq4$, $s\geq1$,

3) $n=3^l\cdot5^r$,

4) $n=3^l\cdot5^r\cdot7^s$, where $l\equiv r\mod 2$, $s\geq1$.\qed
\end{proposition}

\begin{proposition}\label{prop_2_Carm_11}
Let $n\in\mathbb N$ be a composite $11$-smooth number which is not $7$-smooth and not a prime power. Then $n$ is $2$-Carmichael if and only if $n$ belongs to one of the following families:

1) $n=2^k\cdot3^l\cdot5^r\cdot11^t$, where $k\geq3$, $r\geq1$,

2) $n=2^k\cdot3^l\cdot5^r\cdot7^s\cdot11^t$, where $k\geq4$, $r\geq1$, $s\geq1$,

3) $n=2^k\cdot3^l\cdot11^t$, where $k\geq3$, $k\equiv l\mod 2$,


4) $n=2^k\cdot3^l\cdot7^s\cdot11^t$, where $k\geq4$, $s\geq1$, $k+l+s$ is even,

5) $n=3^l\cdot5^r\cdot11^t$, where $r\geq1$,

6) $n=3^l\cdot5^r\cdot7^s\cdot11^t$, where $r\geq1$, $s\geq1$, $l+r+t$ is even,

7) $n=3^l\cdot11^t$, where $l$ is even,


8) $n=3^l\cdot7^s\cdot11^t$, where 
$s\geq1$, $l\equiv s\equiv t\mod 2$.\qed
\end{proposition}

Now consider $m = 3, 4$. We restrict ourselves to composite numbers of the form $n = 2^k3^l$.

\begin{proposition}\label{prop_3_Carm}
Let $n = 2^k3^l$, where $k, l\geq1$. Then $n$ is $3$-Carmichael if and only if $k\geq2$, and $(k, l)$ belongs to one of the following families:

1) $k\equiv 0\mod 12$, $l\equiv 0, \pm2, 3\mod 6$,

2) $k\equiv \pm2\mod 12$, $l\equiv \pm4\mod 6$,

3) $k\equiv \pm4\mod 12$, $l\equiv 0, \pm1, \pm2, \pm4\mod 6$,

4) $k\equiv 6\mod 12$, $l\equiv 0\mod 3$.
\end{proposition}

\begin{proposition}\label{prop_4_Carm}
Let $n = 2^k3^l$, where $k, l\geq1$. Then $n$ is $4$-Carmichael if and only if $k\geq3$, and $(k, l)$ belongs to one of the families 1) -- 4) in Proposition \ref{prop_3_Carm} or to one of the following families:

5) $k\equiv \pm1\mod 12$, $l\equiv \pm2\mod 6$,

6) $k\equiv \pm3\mod 12$, $l\equiv 0\mod 3$,

7) $k\equiv \pm5\mod 12$, $l\equiv \pm4\mod 6$.
\end{proposition}

\noindent\emph{Proof of Propositions \ref{prop_3_Carm} and \ref{prop_4_Carm}.} We have $D_3(2)=3\cdot7$, $D_4(2)=3\cdot5\cdot7$, $D_3(3)=2^3\cdot13$, $D_4(3)=2^4\cdot5\cdot13$, $\nabla_3(n)=2$, $\nabla_4(n)=2\cdot3$. Therefore, $n$ is $3$-Carmichael if and only if $K_3(n)=2n(n^2-1)(n^2+n+1)$ is divisible by $2^3$, $3$, $7$, and $13$, which is equivalent to the conditions $4\mid n$ (i.e., $k\geq2$), $n\equiv\pm1, 2, 4\mod 7$, $n\equiv\pm1, 3, 9\mod 13$. Similarly, $n$ is $4$-Carmichael if and only if $K_4(n)=6n(n^2-1)(n^2+n+1)(n^2+1)$ is divisible by $2^4$, $3$, $5$, $7$, and $13$, which is equivalent to $8\mid n$ (i.e., $k\geq3$), $n\equiv\pm1, 2, 4\mod 7$, $n\equiv\pm1, 3, 9, \pm5\mod 13$. Since $|\mathbb F_7^\times|=6$, $|\mathbb F_{13}^\times|=12$, and $3\mod 13$ is of order $3$, we see that the conditions on $n$ modulo $7$ and $13$ depend only on $k\mod 12$, $l\mod 6$. The corresponding values of $k\mod 12$, $l\mod 6$ are obtained by a direct calculation.\qed

\medskip

Now we describe the general pattern of the distribution of $m$-Carmichael numbers with prescribed prime factors.

Let $P$ be a finite nonempty subset of primes. Denote by $D_m(P)$ the least common multiple of $D_m(p)$ for all $p\in P$.

Let us say that $n\in\mathbb N$ is a \emph{$P$-number}, if $n$ is divisible precisely by the primes in $P$. By Theorem \ref{thm_Korselt}, a $P$-number $n$ is $m$-Carmichael if and only if $D_m(P)\mid K_m(n)$.

Further, write $D_m(P)=D_m'(P)D_m''(P)$, where $D_m'(P)$ is a product of primes in $P$, and $D_m''(P)$ is coprime to all $p\in P$. Then a $P$-number $n$ is $m$-Carmichael if and only if $D_m'(P)\mid K_m(n)$ and $D_m''(P)\mid K_m(n)$.

First, notice that, since $n\mid K_m(n)$, $\nabla_m(P):=\nabla_m(n)$ depends only on $P$, and $D_m'(P)$ is coprime to $D_m(n)$, the condition $D_m'(P)\mid K_m(n)$ is satisfied if $\ord_pn\geq \ord_pD_m'(P)-\ord_p\nabla_m(P)$ for all primes $p\in P$.

Secondly, since a $P$-number $n$ is, by construction, invertible modulo $D_m''(P)$, we see that the condition $D_m''(P)\mid K_m(n)$ depends only on the values of $\ord_pn\mod\lambda(D_m''(P))$ for $p\in P$. Here $\lambda$ is the Carmichael function, i.e., $\lambda(D_m''(P))$ is the exponent of the group $(\mathbb Z/D_m''(P)\mathbb Z)^\times$. Moreover, if $p\in P$, let us denote by $v_{m, P}(p)$ the order of $p\mod D_m''(P)$ in the group $(\mathbb Z/D_m''(P)\mathbb Z)^\times$. Then the condition $D_m''(P)\mid K_m(n)$ depends only on the values of $\ord_pn\mod v_{m, P}(p)$ for $p\in P$.

Thus, we get the following

\begin{theorem}\label{thm_P}
For any $p\in P$, the set of $P$-numbers which are $m$-Car\-mi\-chael is invariant under multiplication by $p^{v_{m, P}(p)}$.\qed
\end{theorem}

\begin{corollary}\label{cor_inf}
Assume that there exists a $P$-number which is $m$-Car\-mi\-chael. Then there are infinitely many of them.\qed
\end{corollary}

\begin{remark}
Corollary \ref{cor_inf} follows also from Corollary \ref{cor_powers}.
\end{remark}

All propositions of this section can be viewed as examples to Theorem \ref{thm_P}. E.g., for $m=3$ and $P=\{2,3\}$ we have $D_m'(P)=2^3\cdot3$, $D_m''(P)=7\cdot13$, and $\lambda(7\cdot13)=\lcm(6,12)=12$, $v_{m, P}(2)=12$, $v_{m, P}(3)=6$, which is in accordance with Proposition \ref{prop_3_Carm}.


\section{Concluding remarks}\label{sec_4}
We list some natural questions that remain open.
\begin{itemize}
\item For what 
$m$ and $P$ are there $P$-numbers which are $m$-Carmichael?
\item Are there squarefree $m$-Carmichael numbers for $m\geq3$?
\end{itemize}

\begin{remark}
One can consider an analogous notion for other affine group schemes of finite type defined over $\mathbb Z$. Namely, if $G$ is such a group scheme, $K_G(p)$ the exponent of the group $G(\mathbb F_p)$, and $K_G(n)$ its reasonable extra\-po\-lation to all $n\in\mathbb N$, then one can consider $G$-Carmichael numbers, i.e., composite $n\in\mathbb N$ such that $g^{K_G(n)}=1$ for all $g\in G(\mathbb Z/n\mathbb Z)$.
\end{remark}

\appendix
\section{Numerical experiments}\label{sec_3}
We also calculate, via brute force, all $m$-Carmichael numbers up to $10^5$ for $2\leq m\leq10$. Let us call an $m$-Carmichael number \emph{nontrivial} if it is not a prime power. There are $1330$ nontrivial $2$-Carmichael numbers, $44$ nontrivial $3$-Carmichael numbers, and $28$ nontrivial $4$-Carmichael numbers on the researched interval.
There are none nontrivial $m$-Carmichael numbers for $5\leq m\leq10$ on the researched interval.

Among $16$ Carmichael numbers less than $10^5$, four, namely
\begin{gather*}
1729=7\cdot13\cdot19, \quad 2465=5\cdot17\cdot29, \\
6601=7\cdot23\cdot41, \quad 41041=7\cdot11\cdot13\cdot41,
\end{gather*}
are $2$-Carmichael. None of these Carmichael numbers are $m$-Carmichael for $3\leq m\leq10$. Moreover, none of $m$-Carmichael numbers for $3\leq m\leq10$ on the researched interval are squarefree.

There are $18$ numbers on the researched interval, namely

\medskip
\begin{tabular}{lll}
$48=2^4\cdot3$, & $144=2^4\cdot3^2$, & $1296=2^4\cdot3^4$, \\
$1728=2^6\cdot3^3$, & $2304=2^8\cdot3^2$, & $5760=2^7\cdot3^2\cdot5$, \\
$9216=2^{10}\cdot3^2$, & $11664=2^4\cdot3^6$, & $20736=2^8\cdot3^4$, \\
$25600=2^{10}\cdot5^2$, & $27000=2^3\cdot3^3\cdot5^3$, & $30720=2^{11}\cdot3\cdot5$, \\
$34992=2^4\cdot3^7$, & $36864=2^{12}\cdot3^2$, & $46656=2^6\cdot3^6$, \\
$62208=2^8\cdot3^5$, & $96768=2^9\cdot3^3\cdot7$, & $99225=3^4\cdot5^2\cdot7^2$,
\end{tabular}
\medskip

\noindent that are nontrivial $m$-Carmichael numbers for all $m\in\{2, 3, 4\}$, and one number, $22815=3^3\cdot5\cdot13^2$, that is nontrivial $3$-Carmichael and $4$-Carmichael, but not $2$-Carmichael.

Also, on the researched interval there are $19$ numbers, namely

\medskip
\begin{tabular}{lll}
$160=2^5\cdot5$, & $448=2^6\cdot7$, & $704=2^6\cdot11$, \\
$800=2^5\cdot5^2$, & $1056=2^5\cdot3\cdot11$, & $2640=2^4\cdot3\cdot5\cdot11$, \\
$3136=2^6\cdot7^2$, & $5929=7^2\cdot11^2$, & $7744=2^6\cdot11^2$, \\
$18144=2^5\cdot3^4\cdot7$, & $20000=2^5\cdot5^4$, & $21952=2^6\cdot7^3$, \\
$28672=2^{12}\cdot7$, & $29952=2^8\cdot3^2\cdot13$, & $31744=2^{10}\cdot31$, \\
$34496=2^6\cdot7^2\cdot11$, & $39424=2^9\cdot7\cdot11$, & $45056=2^{12}\cdot11$, \\
$85184=2^6\cdot11^3$,
\end{tabular}
\medskip

\noindent that are nontrivial $2$- and $3$-Carmichael, but not $4$-Carmichael;
$8$ numbers, namely
$216=2^3\cdot3^3$, $1152=2^7\cdot3^2$, $2592=2^5\cdot3^4$, $4000=2^5\cdot5^3$, $5832=2^3\cdot3^6$, $13824=2^9\cdot3^3$, $28800=2^7\cdot3^2\cdot5^2$, $73728=2^{13}\cdot3^2$,
that are nontrivial $2$- and $4$-Carmichael, but not $3$-Carmichael;
$6$ numbers, namely
$324=2^2\cdot3^4$, $900=2^2\cdot3^2\cdot5^2$, $1404=2^2\cdot3^3\cdot13$, $39204=2^2\cdot3^4\cdot11^2$, $74088=2^3\cdot3^3\cdot7^3$, $74536=2^3\cdot7\cdot11^3$,
that are nontrivial $3$-Carmichael, but not $2$- or $4$-Carmichael.
Finally, one number, $26112=2^9\cdot3\cdot17$, is nontrivial $4$-Carmichael, but not $2$- or $3$-Carmichael.

	
Below we list all nontrivial $2$-Carmichael numbers up to $3000$ that are not treated by Propositions \ref{prop_2_Carm_7} and \ref{prop_2_Carm_11} (i.e., not $11$-smooth).

\medskip

\begin{tabular}{lll}
$104=2^3\cdot13$ & $171=3^2\cdot19$ & $195=3\cdot5\cdot13$ \\
$273=3\cdot7\cdot13$ & $351=3^3\cdot13$ & $435=3\cdot5\cdot29$ \\
$455=5\cdot7\cdot13$ & $609=3\cdot7\cdot29$ & $615=3\cdot5\cdot41$ \\
$624=2^4\cdot3\cdot13$ & $665=5\cdot7\cdot19$ & $715=5\cdot11\cdot13$ \\
$736=2^5\cdot23$ & $759=3\cdot11\cdot23$ & $832=2^6\cdot13$ \\
$855=3^2\cdot5\cdot19$ & $903=3\cdot7\cdot43$ & $1001=7\cdot11\cdot13$ \\
$1015=5\cdot7\cdot29$ & $1045=5\cdot11\cdot19$ & $1071=3^2\cdot7\cdot17$ \\
$1088=2^6\cdot17$ & $1183=7\cdot13^2$ & $1216=2^6\cdot19$ \\
$1265=5\cdot11\cdot23$ & $1352=2^3\cdot13^2$ & $1377=3^4\cdot17$ \\
$1431=3^3\cdot53$ & $1456=2^4\cdot7\cdot13$ & $1520=2^4\cdot5\cdot19$ \\
$1539=3^4\cdot19$ & $1560=2^3\cdot3\cdot5\cdot13$ & $1595=5\cdot11\cdot29$ \\
$1625=5^3\cdot13$ & $1729=7\cdot13\cdot19$ & $1856=2^6\cdot29$ \\
$1881=3^2\cdot11\cdot19$ & $1911=3\cdot7^2\cdot13$ & $1984=2^6\cdot31$ \\
$2001=3\cdot23\cdot29$ & $2009=7^2\cdot41$ & $2015=5\cdot13\cdot31$ \\
$2080=2^5\cdot5\cdot13$ & $2211=3\cdot11\cdot67$ & $2255=5\cdot11\cdot41$ \\
$2365=5\cdot11\cdot43$ & $2375=5^3\cdot19$ & $2457=3^3\cdot7\cdot13$ \\
$2465=5\cdot17\cdot29$ & $2535=3\cdot5\cdot13^2$ & $2565=3^3\cdot5\cdot19$ \\
$2624=2^6\cdot41$ & $2639=7\cdot13\cdot29$ & $2736=2^4\cdot3^2\cdot19$ \\
$2808=2^3\cdot3^3\cdot13$ & $2871=3^2\cdot11\cdot29$ & $2912=2^5\cdot7\cdot13$ \\
$2925=3^2\cdot5^2\cdot13$
\end{tabular}

\medskip

We also managed to compute a few larger $m$-Carmichael numbers for $m\geq5$. For instance, $2^{22}\cdot3^2$ is $2$- (by Proposition \ref{prop_2_Carm_7}), $3$- (by Proposition \ref{prop_3_Carm}), $4$- (by Proposition \ref{prop_4_Carm}), $5$- and $6$-Carmichael, but not $7$- or $8$-Carmichael. Similarly, $2^{286}\cdot3^{36}$ is $2$-, $6$-, $7$-, and $8$-Carmichael, but not $3$-, $4$-, or $5$-Car\-mi\-chael.



\end{sloppy}
\end{document}